\RequirePackage{fix-cm}

\documentclass[smallextended,referee,envcountsect]{svjour3}
\smartqed

\usepackage{amssymb,amsfonts,amstext,amsmath}
\usepackage{mathptmx}
\usepackage{xcolor}
\usepackage{cite}
\usepackage{leading}
\usepackage{graphicx,geometry}
\newtheorem{thm}{Theorem}[section]

\newcommand{\disp}{\displaystyle}

\def\Om{{\Omega}}

\def\Om{{\Omega}}

\def\jnt{\disp\int}
\def\jjnt{\jnt\!\!\!\jnt}

\journalname{Journal of Optimization Theory and Applications}

\begin{document}

\title{Hierarchical Control for the Wave Equation with a Moving Boundary\thanks{Communicated by Roland  Glowinski}}

 \author{Isa\'{\i}as Pereira  de Jesus}
          \institute{ Isa\'{\i}as Pereira  de Jesus ({\bf Corresponding Author})   \at DM, Universidade Federal do Piau\'{\i}, Teresina, PI 64049-550, Brazil.\\
              \email{isaias@ufpi.edu.br}
         }
\date{Received: date / Accepted: date}

\maketitle

\begin{abstract} This paper  addresses the study of the  hierarchical control for the one-dimensional
wave equation in intervals with a moving boundary. This equation models the motion of a string where an endpoint is fixed and the other one is moving. When the
speed of the moving endpoint is less than the characteristic speed, the controllability of this equation is established. We assume that we can act on the
dynamic of the system by a hierarchy of controls. According to the formulation
given by H. von Stackelberg (Marktform und Gleichgewicht, Springer, Berlin, 1934), there are local controls called followers and
global controls called leaders. In fact, one considers situations where there
are two cost (objective) functions. One possible way is to cut the control into two
parts, one being thought of as ``the leader" and the other one as ``the follower".
This situation is studied in the paper, with one of the cost functions being of the
controllability type. We present the following results: the existence
and uniqueness of Nash equilibrium, the approximate controllability  with respect
to the leader control, and the optimality system for the leader control.



\end{abstract}

\section{Introduction}

In classical control theory, we usually find a state equation or system and one
control with the mission of achieving a  predetermined goal. Frequently (but not
always), the goal is to minimize a cost functional in a prescribed family of
admissible controls.

A more interesting situation arises when several (in general, conflictive or
contradictory) objectives are considered. This may happen, for example, if the
cost function is the sum of several terms and it is not clear how to average. It is also expected to have more than one control acting on the equation. In
these cases we are led to consider multi-objective control  problems.

In contrast to the mono-objective case and depending on the characteristics of the problem, there are many strategies for the choice of
good controls. Moreover, these strategies can be cooperative (when the controls mutually
cooperate in order to achieve some goals) or noncooperative.

There exist several equilibrium concepts for multi-objective problems, with origin
in game theory, mainly motivated by economics. Each of them determines a
strategy. Let us mention the noncooperative optimization strategy proposed
by Nash \cite{N}, the Pareto cooperative strategy \cite{P}, and the Stackelberg
hierarchical-cooperative strategy \cite{S}.

In the context of game theory,  Stackelberg's strategy is normally applied
within games where some players are in a better position than others. The
dominant players are called the leaders and the sub-dominant players the
followers. One situation where this concept is very convincing is when the
players choose their strategies one after another and the player who does the
first move has an advantage. This interpretation however is not so convincing in
continuous time where players choose their strategies at each time, but one
could think of the  situation  where one player has a larger information structure. From a
mathematical point of view,  Stackelberg's  concept is successful because it leads to
results.

In the context of the control of PDEs, a relevant question is whether one is able
to steer the system to a desired state (exactly or approximately) by applying
controls that correspond to one of these strategies.

This paper was inspired by ideas of J.-L. Lions
\cite{L1}, where we  investigate a  similar question of hierarchical control
employing the Stackelberg strategy in the case of time dependent domains.

We list below some related works on this subject up to date.

\textbf{$\bullet$}  The papers by Lions  \cite {L10, L12}, where the author gives some results
concerning Pareto and Stackelberg strategies, respectively.

\textbf{$\bullet$} The paper by  D\'iaz and Lions \cite {DL}, where the approximate
controllability of a system is established following a Stackelberg-Nash
strategy and the extension in D\'iaz \cite{D1}, that provides a
characterization of the solution by means of Fenchel-Rockafellar duality
theory.

\textbf{$\bullet$} The papers \cite {RA1, RA2}, where Glowinski, Ramos and Periaux   analyze the
Nash equilibrium for constraints given by linear parabolic and Burger's equations from the mathematical and numerical viewpoints.

\textbf{$\bullet$} The Stackelberg-Nash strategy for the Stokes systems has been studied
 by Gonz\'alez, Lopes and Medar in \cite {GO}.

\textbf{$\bullet$} In Limaco, Clark and Medeiros \cite {LI}, the authors present the
Stackelberg-Nash equilibrium in the context of linear heat equation in non-cylindrical domains.

\textbf{$\bullet$} The paper by Araruna, Fern{\'a}ndez-Cara and Santos \cite{AR}, where the authors study  the Stackelberg-Nash exact
controllability for linear and semilinear parabolic equations.

\textbf{$\bullet$} The paper by Ramos   and Roubicek  \cite{RR},  where the existence of a Nash equilibrium
is proved for a nonlinear distributed parameter predator-prey system and a
conceptual approximation algorithm is proposed and analyzed.

In this paper, we present the following results: the existence and uniqueness of
Nash equilibrium, the approximate controllability with respect to the leader
control, and the optimality system for the leader control.

The remainder of the paper is organized as follows. In Section \ref{sec2}, we present the
problem. Section 3 is devoted to establish the existence and uniqueness of Nash equilibrium. In Section \ref{sec4}, we study the approximate controllability
with respect to the leader control. In Section \ref{sec5}, we present the optimality system for the leader control. Finally,
in the Section \ref{sec6} we present the conclusions.

\section{Problem Formulation}\label{sec2}
As in \cite{Cui2}, given  $ T > 0 $, we consider the non-cylindrical domain defined
by
$$
\disp \widehat{Q}:= \left\{ (x,t) \in \mathbb{R}^2;\; 0 < x <  \alpha _k(t),\;  \; t \in  ]0,T[ \right\},
$$
where
$$
\disp \alpha_k(t):= 1 + kt, \;\;\;\;\;\; \; 0 < k < 1.
$$

Its lateral boundary  is defined by $\disp \widehat{\Sigma}:=\widehat{\Sigma}_0
\cup \widehat{\Sigma}_0^*$, with
$$\disp \widehat{\Sigma}_{0}:= \{(0,\,t);\;   t \in ]0,T[ \} \;\;\;\;\mbox{ and } \;\;\;\; \widehat{\Sigma}_{0}^* =\widehat{\Sigma} \backslash \widehat{\Sigma}_0:= \{(\alpha_k(t),t);\;   t \in  ]0,T[ \}.$$
We also denote by $\Om_t$ and $\Om_0$ the intervals $\disp  ]0, \alpha_k(t)[$
and $\disp ]0, 1[$, respectively.

Consider the following  wave equation in the non-cylindrical domain
$\widehat{Q}$:
\begin{equation} \label{eq1.3}
\begin{array}{l}
\displaystyle u'' - u_{xx} = 0 \ \ \mbox{ in } \ \ \widehat{Q},\\[11pt]
\disp u(x,t) = \left\{
\begin{array}{l}
\widetilde{w} \ \ \mbox{ on } \ \ \widehat{\Sigma}_{0},\\[11pt]
0 \ \ \mbox{ on } \ \ \widehat{\Sigma}_0^*,
\end{array}
\right.\\[11pt]
\disp u(x,0) = u_0(x), \;\; u'(x,0) = u_1(x) \ \mbox{ in }\; \Om_0,
\end{array}
\end{equation}
where $u$ is the state variable, $\widetilde{w}$ is the control variable and $( u_0(x),
u_1(x)) \in L^2(0, 1) \times H^{-1}(0, 1)$. By $\disp u'=u'(x,t)$ we represent the
derivative $\disp \frac{\partial u}{\partial t}$ and by $\disp u_{xx}=u_{xx}(x,t)$ the second order partial derivative $\disp \frac{\partial^2 u}{\partial x^2}$. Equation \eqref{eq1.3} models the motion of a string with a fixed endpoint and a
moving one. The constant $k$ is called the {\it speed of the moving endpoint.}

In spite of a vast literature on the controllability problems of the wave equation in cylindrical domains, there are only
a few works dealing with non-cylindrical case. We refer to \cite {Ar, Bar, Cui1, Cui,
Je, Mi1, Mi} for some known results in this direction.

In this article, motivated by the arguments contained in the work of J.-L. Lions
\cite{L1}, we investigate a  similar question of hierarchical control for equation
(\ref{eq1.3}), employing the Stackelberg strategy in the case of time
 dependent domains.

Following the work of J.-L. Lions \cite{L1}, we divide $\disp \widehat{\Sigma}_0$
into two disjoint  parts
\begin{equation}\label{decomp0}
\disp \widehat{\Sigma}_0=\disp \widehat{\Sigma}_1 \cup \disp \widehat{\Sigma}_2,
\end{equation}
 and consider
\begin{equation} \label{decomp}
\disp \widetilde{w}=\{\widetilde{w}_1, \widetilde{w}_2\}, \;\; \widetilde{w}_i:=\mbox{control function in } \; L^2(\widehat{\Sigma}_i), \;i=1,2.
\end{equation}

Thus, we observe  that the system \eqref{eq1.3} can be rewritten as follows:
\begin{equation} \label{eq1.3.1}
\begin{array}{l}
\displaystyle u''- u_{xx} = 0 \ \ \mbox{ in } \ \ \widehat{Q},\\[5pt]
\disp u(x,t) = \left\{
\begin{array}{l}
\widetilde{w}_1 \ \ \mbox{ on } \ \ \widehat{\Sigma}_{1},\\
\widetilde{w}_2 \ \ \mbox{ on } \ \ \widehat{\Sigma}_{2},\\
0 \ \ \mbox{ on } \ \ \widehat{\Sigma}\backslash \widehat{\Sigma}_0,
\end{array}
\right.\\[5pt]
\disp u(x,0) = u_0(x), \;\; u'(x,0) = u_1(x) \ \mbox{ in }\; \Om_0.
\end{array}
\end{equation}
In the decomposition \eqref{decomp0}, \eqref{decomp} we establish a hierarchy.
We think of $\widetilde{w}_1$ as being the  ``main" control, the  leader, and we
think of $\widetilde{w}_2$ as the  follower,  in Stackelberg terminology.

Associated with the solution $u=u(x,t)$ of \eqref{eq1.3.1}, we will consider the
(secondary) functional
\begin{equation}\label{sfn}
\disp \widetilde{J}_2(\widetilde{w}_1, \widetilde{w}_2):=
\displaystyle\frac{1}{2} \displaystyle\jjnt_{\widehat{Q}} \left(u(\widetilde{w}_1, \widetilde{w}_2)-\widetilde{u}_2\right)^2 dx dt +
\displaystyle\frac{\widetilde{\sigma}}{2} \int_{\widehat{\Sigma}_2} \widetilde{w}_2^2\;d\widehat{\Sigma},
\end{equation}
and the (main) functional
\begin{equation}\label{mfn}
\disp \widetilde{J}(\widetilde{w}_1):=\frac{1}{2}\int_{\widehat{\Sigma}_1} \widetilde{w}_1^2\;d\widehat{\Sigma},
\end{equation}
where $\widetilde{\sigma}>0$ is a constant and $\widetilde{u}_2$ is a given
function in $L^2(\widehat{Q}).$

\begin{remark}\label{bdfnc} \textrm{From the regularity and uniqueness of the solution  of \eqref{eq1.3.1} (see Remark $4.4$) the cost functionals $\disp \widetilde{J}_2$ and $\disp \widetilde{J}$ are well defined.}
\end{remark}
The control problem that we will consider is as follows: the follower $\disp
\widetilde{w}_2$ assume that  the leader $\disp \widetilde{w}_1$ has made  a
choice. Then, it tries  to find an equilibrium of the cost  $\widetilde{J}_2$ , that is,
it  looks for  a control $\disp \widetilde{w}_2=\mathfrak{F}(\widetilde{w}_1)$
(depending on $\disp \widetilde{w}_1$), satisfying:
\begin{equation}\label{n1.1}
\disp \widetilde{J}_2(\widetilde{w}_1, \widetilde{w}_2) \leq \widetilde{J}_2(\widetilde{w}_1,  \widehat{w}_2), \;\;\; \forall\; \widehat{w}_2  \in L^2(\widehat{\Sigma}_2).
\end{equation}

The control $ \widetilde{w}_2 $, solution of \eqref{n1.1}, is called {\it Nash
equilibrium} for the cost $\widetilde{J}_2 $ and it depends on $ \widetilde{w}_1$
(cf. Aubin \cite{A}).

\begin{remark}\label{r2} \textrm{In another way, if the leader $ \widetilde{w}_1$ makes a choice, then the follower $ \widetilde{w}_2$ makes also a choice, depending on  $\widetilde{w}_1$, which minimizes the cost $\widetilde{J}_2$, that is,
\begin{equation}\label{son}
\disp \widetilde{J}_2(\widetilde{w}_1, \widetilde{w}_2)= \inf_{\widehat{w}_2 \in L^2(\widehat{\Sigma}_2)} \widetilde{J}_2(\widetilde{w}_1, \widehat{w}_2).
\end{equation}
This is equivalent to \eqref{n1.1}. This process is called  Stackelberg-Nash
strategy; see D\'iaz and Lions \cite{DL}.}
\end{remark}

After this, we consider the state $\disp u\left(\widetilde{w}_1,
\mathfrak{F}(\widetilde{w}_1)\right)$ given by the solution of
\begin{equation} \label{eq1.3.1F}
\begin{array}{l}
\displaystyle u''- u_{xx} = 0 \ \ \mbox{ in } \ \ \widehat{Q},\\[3pt]
\disp u(x,t) = \left\{
\begin{array}{l}
\widetilde{w}_1 \ \ \mbox{ on } \ \ \widehat{\Sigma}_{1},\\
\mathfrak{F}(\widetilde{w}_1) \ \ \mbox{ on } \ \ \widehat{\Sigma}_{2},\\
0 \ \ \mbox{ on } \ \ \widehat{\Sigma}\backslash \widehat{\Sigma}_0,
\end{array}
\right.\\[3pt]
\disp u(x,0) = u_0(x), \;\; u'(x,0) = u_1(x) \ \mbox{ in }\; \Om_0.
\end{array}
\end{equation}

We will look for any optimal control $\disp \widetilde{w}_1$ such that
\begin{equation}\label{ocn}
\disp \widetilde{J}(\widetilde{w}_1, \mathfrak{F}(\widetilde{w}_1))= \inf_{\overline{w}_1 \in L^2(\widehat{\Sigma}_1)} \widetilde{J}(\overline{w}_1, \mathfrak{F}(\widetilde{w}_1)),
\end{equation}
subject to the following restriction of the approximate controllability type
\begin{equation}\label{apcn}
\begin{array}{l}
\disp\left( u(x, T; \widetilde{w}_1, \mathfrak{F}(\widetilde{w}_1)), u'(x, T; \widetilde{w}_1, \mathfrak{F}(\widetilde{w}_1))\right)\in B_{L^2(\Omega_t)}(u^0,\rho_0) \times B_{H^{-1}(\Omega_t)}(u^1,\rho_1),
\end{array}
\end{equation}
where $\disp B_X(C, r)$ denotes the ball in $X$ with centre $C$ and radius $r$.

To explain this optimal problem, we are going to consider the following
sub-problems:

\textbf{$\bullet$ Problem 1} Fixed any leader control $\widetilde{w}_1$,  find the
follower control $\disp \widetilde{w}_2=\mathfrak{F}(\widetilde{w}_1)$ (depending
on $\disp \widetilde{w}_1$) and the associated state $u$, solution of
\eqref{eq1.3.1} satisfying the condition \eqref{son} (Nash equilibrium)  related to
$\widetilde{J}_2$, defined in \eqref{sfn}.

\textbf{$\bullet$ Problem 2} Assuming that the existence of the Nash equilibrium
$\disp \widetilde{w}_2$ was proved, then when $\widetilde{w}_1$ varies in
$L^2(\widehat{\Sigma}_1)$, prove that the solutions $\disp\left(u(x, t;
\widetilde{w}_1, \widetilde{w}_2), u'(x, t; \widetilde{w}_1, \widetilde{w}_2)\right)$ of
the state equation \eqref{eq1.3.1}, evaluated at $t = T$, that is, $\disp\left(u(x, T;
\widetilde{w}_1, \widetilde{w}_2), u'(x, T; \widetilde{w}_1, \widetilde{w}_2)\right)$,
generate a dense subset of $\disp L^2(\Omega_t) \times H^{-1}(\Omega_t)$.

\begin{remark}\label{r1} \textrm{By the linearity of system \eqref{eq1.3.1F}, without loss of generality we may assume that $u_0=0=u_1$.}
\end{remark}
Following the work of J.-L. Lions \cite{L1}, we divide $\disp {\Sigma}_0$ into two
disjoint parts
\begin{equation}\label{decomp0.1}
\disp {\Sigma}_0=\disp {\Sigma}_1 \cup \disp {\Sigma}_2,
\end{equation}
and consider
\begin{equation} \label{decomp.2}
\disp w = \{w_1,w_2\}, \;\; {w}_i:=\mbox{control function in } \; L^2({\Sigma}_i), \;i=1,2.
\end{equation}
We can also write
\begin{equation} \label{decomp 2.A}
\disp w= w _1+ w_2, \; \mbox{ with } \; \disp {\Sigma}_0=\disp {\Sigma}_1 = \disp {\Sigma}_2.
\end{equation}

Note that when $\disp (x,t)$ varies in $\disp \widehat{Q}$ the point $\disp (y,t)$,
with $\disp y=\frac{x}{\alpha_k(t)}$, varies in $\disp Q=\Om \times ]0,T[,$ where
$\Om := ]0,1[$. Then, the application $$\disp \zeta:\widehat{Q} \to Q,
\;\;\;\zeta(x,t)=(y,t)$$ is of class $\disp C^2$ and the inverse $\disp \zeta^{-1}$ is
also of class $\disp C^2$. Therefore the change of variables $\disp u(x,t)=v(y,t)$,
transforms the initial-boundary value problem \eqref{eq1.3.1} into the equivalent
system
\begin{equation} \label{eq1.14}
\begin{array}{l}
\disp v'' + Lv = 0  \ \ \mbox{ in } \ \ Q,\\ [13pt]
\disp v(y,t) = \left\{
\begin{array}{l}
w_1 \ \ \mbox{on} \ \ \Sigma_1,\\
w_2 \ \ \mbox{on} \ \ \Sigma_2,\\
0 \ \ \mbox{on} \ \ \Sigma\backslash\Sigma_0,
\end{array}
\right. \\[5pt]
\disp v(y,0) = v_0(y), \; v'(y,0) = v_1(y), \;\; y \in \Om,
\end{array}
\end{equation}
where
\begin{equation*}
\begin{array}{lll}
\disp Lv = - \Big[\frac{\beta_k(y,t)}{\alpha_k(t)}v_{y}\Big]_{y} + \frac{\gamma_k(y)}{\alpha_k(t)}v'_{y}\;\;, \;\; \;
\disp \beta_{k}(y,t) = \frac{1  - k^2y^2}{\alpha_{k}(t)},\;\;\;
\disp \gamma_{k}(y) = -2ky,\\[10pt]
\disp v_{0}(y) = u_{0}(x), \;\;\;
\disp v_{1}(y) = u_{1}(x) + kyu_{x}(0), \;\;\;
\disp \Sigma= {\Sigma}_0 \cup \Sigma^*_0, \\[10pt]
\disp {\Sigma}_0 = \{(0,t) : 0 < t < T \}, \;\;\;
\disp {\Sigma}^*_0 = \{(1,t) : 0 < t < T \}, \;\;\;
\disp \disp {\Sigma}_0 = \Sigma_1  \cup \Sigma_2.
\end{array}
\end{equation*}

We consider the coefficients of the operator $L$  satisfying  the following
conditions:
\begin{itemize}
\item[(H1)] $\beta_{k}(y,t)  \in C^{1}(\overline{Q});$
\item[(H2)] $\gamma_{k}(y)  \in  W^{1,\infty}(\Om).$
\end{itemize}

In this way, it is enough to investigate the control problem for the equivalent
problem \eqref{eq1.14}.

$\bullet$ \textbf{ Cost functionals in the cylinder $Q$.}  From the diffeomorphism
$\disp \zeta,$ which transforms $\widehat{Q}$ into $Q$, we transform the cost
functionals $\disp \widetilde{J}_2, \disp \widetilde{J}$ into the cost functionals
$J_2, J$ defined by
\begin{equation} \label{eq3.9}
\disp J_2(w_1,w_2):= \frac{1}{2}\,\int_{0}^{T}\int_{\Omega}\alpha_{k}(t)[v(w_1,w_2) - v_2(y,t)]^2dy\,dt + \frac{\sigma}{2}\,\int_{\Sigma_2}w_{2}^{2}\,d\Sigma
\end{equation}
and
\begin{equation} \label{eq3.7}
\disp J(w_1):= \frac{1}{2}\,\int_{\Sigma_1}w_1^{2}\,d\Sigma,
\end{equation}
where $\sigma > 0$ is a constant and $v_2(y,t)$ is a given function in $L^2(\Om
\times (0,T)).$
\begin{remark}\label{rsol} \textrm{Using similar technique as  in \cite{Mi}, we can prove the  following:
For each $v_0 \in L^2(\Om)$, \linebreak $v_1 \in H^{-1}(\Om)$ and $\disp w_i \in
L^2(\widehat{\Sigma}_i), \;i=1,2,$ there exists exactly one  solution $v$ to
\eqref{eq1.14} in the sense of a transposition, with  $\disp v \in C\big( [0,T]
;L^2(\Om)) \cap C^{1}\big( [0,T] ; H^{-1}(\Om))$. Thus, in particular, the cost
functionals $\disp {J}_2$ and $\disp {J}$ are well defined.
Using the diffeomorphism  $\disp \zeta^{-1}(y,t)=(x,t)$, from $Q$ onto
$\widehat{Q}$, we obtain a unique global weak solution $\disp u$ to the
problem \eqref{eq1.3.1} with the regularity $\disp u \in C\big( [0,T] ; L^2(\Om_t))
\cap C^{1}\big( [0,T] ; H^{-1}(\Om_t))$.
}
\end{remark}

Associated with the functionals $\disp J_2$ and $\disp J$ defined above, we will
consider the following  sub-problems:

\textbf{$\bullet$ Problem 3} Fixed any leader control ${w}_1$,  find the  follower
control $\disp {w}_2$ (depending on $\disp {w}_1$) and the associated state $v$
solution of \eqref{eq1.14} satisfying (Nash equilibrium)
 \begin{equation}\label{soncil}
\disp {J}_2({w}_1, {w}_2)= \inf_{\widehat{w}_2 \in L^2({\Sigma}_2)} {J}_2({w}_1, \widehat{w}_2),
\end{equation}
related to ${J}_2$ defined in \eqref{eq3.9}.

\textbf{$\bullet$ Problem 4} Assuming that the existence of the Nash equilibrium
$\disp {w}_2$ was proved, then when ${w_1}$ varies in $L^2({\Sigma}_1)$, prove
that the solutions $\disp\left(v(y, t; {w_1}, {w}_2), v'(y, t; {w_1}, {w}_2)\right)$ of
the state equation \eqref{eq1.14}, evaluated at $t = T$, that is, $\disp\left(v(y, T;
{w_1}, {w}_2), v'(y, T; {w_1}, {w}_2)\right)$, generate a dense subset of $\disp
L^2(\Om) \times H^{-1}(\Om)$.

 \section{Nash Equilibrium}\label{sec3}
In this section, fixed any leader control $\disp \ w_1 \in L^2(\Sigma_1)$ we
determine the existence and uniqueness of solutions to the problem
\begin{equation} \label{eq3.10}
\begin{array}{l}
\displaystyle\inf_{w_2 \in L^2(\Sigma_2)}J_2(w_1,w_2),
\end{array}
\end{equation}
and a characterization of this solution in terms of an adjoint system.

In fact, this is a classical type problem in the control of distributed systems (cf.
J.-L. Lions \cite{L3}). It admits an unique solution
\begin{equation} \label{eq3.11}
\disp w_2 = \mathfrak{F}(w_1).
\end{equation}
The Euler - Lagrange equation for problem \eqref{eq3.10} is given by
\begin{equation} \label{eq3.21}
\int_{0}^{T}\int_{\Omega}{\alpha_{k}(t)}(v-v_2)\widehat{v}dy\,dt + \sigma\int_{\Sigma_2}w_2\widehat{w}_2d\Sigma = 0, \;\;\forall\, \widehat{w}_2 \in L^2(\Sigma_2),
\end{equation}
where $\widehat{v}$ is solution of the following system
\begin{equation} \label{eq3.22}
\begin{array}{l}
\disp \widehat v'' +L\widehat{v} = 0 \ \ \mbox{ in } \ \ Q,\\ [7pt]
\disp \widehat{v} = \left\{
\begin{array}{l}
0 \ \ \mbox{ on } \ \ \Sigma_1, \\
\widehat{w}_2 \ \ \mbox{ on } \ \ \Sigma_2, \\
0 \ \ \mbox{ on } \ \ \Sigma \backslash \left(\Sigma_1 \cup \Sigma_2\right),
\end{array}
\right. \\[13pt]
\disp \widehat{v}(y,0) = 0, \; \widehat{v'}(y,0) = 0, \;\; y \in \Om.
\end{array}
\end{equation}
In order to express \eqref{eq3.21} in a convenient form, we introduce the adjoint
state defined by
\begin{equation}\label{sac}
\begin{array}{l}
p'' + L^{\ast}\,p = \alpha_{k}(t)\left(v - v_2\right) \ \mbox{ in } \ \ Q, \\[5pt]\disp
p(T) = p'(T) = 0, \;\; y \in \Om, \;\;\; p = 0 \ \mbox{ on } \ \Sigma,
\end{array}
\end{equation}
where $L^{\ast}$ is the formal adjoint of the  operator $\disp L$.

Multiplyng (\ref{sac}) by $\widehat{v}$ and integrating by parts, we find
\begin{equation} \label{eq3.33}
\int_{0}^{T}\int_\Om \alpha_{k}(t)(v - v_2)\widehat{v}\,dy\,dt + \int_{\Sigma_2} \frac{1}{\alpha_k^2(t)}\,p_y\,\widehat{w}_2\,d\Sigma = 0,
\end{equation}
so that \eqref{eq3.21} becomes
\begin{equation}\label{ci}
\disp p_y= \sigma \alpha_k^2(t)\,w_2 \ \ \mbox{ on } \ \ \Sigma_2.
\end{equation}

We summarize these results in the following theorem:
\begin{thm}\label{teN}  For each $\disp w_1 \in L^2(\Sigma_1)$ there exists a unique Nash equilibrium $\disp w_2$ in the sense of \eqref{soncil}. Moreover, the follower $\disp w_2$ is given by
\begin{equation}\label{cseg}
\disp \disp w_2 = \mathfrak{F}(w_1)=\frac{1}{\sigma \alpha_{k}^2(t)}\,\;p_y\;\;\mbox{ on }\;\;\Sigma_2,
\end{equation}
where $\disp \{ v,p \}$ is the unique solution of (the optimality system)
\begin{equation} \label{eq3.37}
\begin{array}{l}
\disp v'' + Lv = 0 \ \mbox{ in } \;\; Q, \;\;\;   p'' + L^{\ast}\,p = \alpha_{k}(t)\left(v - v_2\right) \ \mbox{ in } \;\; Q,\\[5pt]
\disp v = \left\{
\begin{array}{l}
w_1 \ \mbox{ on } \ \Sigma_1,\\[5pt]
\disp \frac{1}{\sigma \alpha_{k}^2(t)}\;\,p_y \ \mbox{ on } \ \Sigma_2,\\[5pt]
0 \ \mbox{ on } \ \Sigma \backslash \Sigma_0,
\end{array}
\right.\\[7pt]
p = 0 \ \mbox{ on } \ \Sigma, \;\;\; v(0) = v'(0) = 0, \;\;\; p(T) = p'(T) = 0, \;\; y \in \Om.
\end{array}
\end{equation}
Of course, $\disp \{ v,p \}$ depends on $w_1$:
\begin{equation}\label{cdep}
\disp \{ v,p \} = \{ v(w_1),p(w_1)\}.
\end{equation}
\end{thm}

\section{On the Approximate Controllability}\label{sec4}
Since we have proved the existence, uniqueness and characterization of the
follower $\disp w_2$, the leader $\disp w_1$  now wants  that the solutions $v$
and $v'$, evaluated at time $t=T$, to  be as close as possible to $\disp(v^0,
v^1)$. This will be possible if the system \eqref{eq3.37}  is approximately
controllable. We are looking for
\begin{equation} \label{inf1}
\begin{array}{l}
\displaystyle\inf\, \frac{1}{2\,}\,\int_{\Sigma_1} w_{1}^{2}\,d\Sigma,
\end{array}
\end{equation}
where $\disp w_1$ is subject to
\begin{equation} \label{subj1}
\begin{array}{l}
\disp\left(v(T;{w_1}), v'(T; {w_1})\right)\in B_{L^2(\Om)}(v^0,\rho_0) \times B_{H^{-1}(\Om)}(v^1,\rho_1),
\end{array}
\end{equation}
assuming that  $w_1$ exists, $\rho_0$ and  $\rho_1$ being positive numbers arbitrarily
small and \linebreak $\{v^0, v^1\} \in L^2(\Om) \times H^{-1}(\Om)$.

As in \cite{Cui2}, we assume that
\begin{equation}\label{hT}
T >  T_{k}^{\ast}
\end{equation}
and
\begin{equation}\label{hT10}
0 < k <  1 - \frac{1}{ \sqrt{e}}.
\end{equation}

Now as in the case \eqref{decomp 2.A} and using  Holmgren's Uniqueness
Theorem (cf. \cite{LH}; and see also \cite {Cui2} for additional discussions), the following approximate controllability result holds:
\begin{thm}\label{AC}  Assume that \eqref{hT} and \eqref{hT10} hold. Let us consider $\disp w_1 \in L^2(\Sigma_1)$ and $\disp w_2$ a Nash equilibrium in the sense \eqref{soncil}.
Then $\disp \left(v(T), v'(T)\right)=\left(v(., T, {w_1}, w_2), v'(., T, {w_1},
w_2)\right)$, where $\disp v$ solves the system \eqref{eq1.14}, generates a dense
subset of $\disp L^2(\Om)\times H^{-1}(\Om)$.
\end{thm}
\begin{proof}
 We decompose the solution $\disp (v,p)$ of \eqref{eq3.37}
setting
\begin{equation} \label{eq3.39}
v = v_0 + g, \ \ \ \ \  p = p_0 + q,
\end{equation}
where $v_0$, $p_0$ is given by
\begin{equation} \label{eq3.40}
\begin{array}{l}
\disp v_0'' + L\,v_0 = 0 \ \mbox{ in } \ Q,\\[5pt]\disp
v_0 = \left\{
\begin{array}{l}
0 \ \mbox{ on } \ \Sigma_1,\\[5pt]
\disp \frac{1}{\sigma \alpha_{k}^2(t)}\,({p_0})_{y} \ \mbox{ on } \ \Sigma_2,\\[5pt]\disp
0 \ \mbox{ on } \ \Sigma \backslash \Sigma_0,
\end{array}
\right.\\[5pt]\disp
v_0(0) = v_0'(0) = 0, \;\; y\in \Om,
\end{array}
\end{equation}

\begin{equation} \label{eq3.41}
\begin{array}{l}
\disp p_0'' + L^{\ast}p_0 = \alpha_{k}(t) \left(v_0 - v_2\right) \ \mbox{ in } \ Q, \\[5pt]\disp
p_0 = 0 \ \mbox{ on } \ \Sigma, \;\;\;  p_0(T) = p_0'(T) = 0, \;\; y \in \Om,
\end{array}
\end{equation}
and $\disp\{g,q\}$ is given by
\begin{equation} \label{eq3.42}
\begin{array}{l}
g'' + L\,g = 0 \ \mbox{ in } \ Q,\\[5pt]\disp
g = \left\{
\begin{array}{l}
w_1 \ \mbox{ on } \ \Sigma_1,\\[5pt]
\disp \frac{1}{\sigma \alpha_{k}^2(t)}\,q_{y} \ \mbox{ on } \ \Sigma_2,\\[5pt]\disp
0 \ \mbox{on} \ \Sigma \backslash \Sigma_0,
\end{array}
\right.\\[5pt]\disp
g(0) = g'(0) = 0, \;\;y \in \Om,
\end{array}
\end{equation}

\begin{equation} \label{eq3.43}
\begin{array}{l}
\disp q'' + L^{\ast}q = \alpha_{k}(t) g \ \mbox{ in } \ Q, \\[5pt]\disp
q = 0 \ \mbox{ on } \ \Sigma, \;\;\; q(T) = q'(T) = 0, \;\; y \in \Om.
\end{array}
\end{equation}
We next set
\begin{equation} \label{eq3.44}
\begin{array}{ccll}
A \ : & \! L^2(\Sigma_1) & \! \longrightarrow & \! H^{-1}(\Om)  \times L^2(\Om)\\
& \! w_1  & \!\longmapsto & \! A\,w_1  =  \big\{ g'(T;w_1) + \delta g(T;w_1),\; -g(T;w_1) \big\},
\end{array}
\end{equation}

which defines
$$A \in \mathcal{L}\left( L^2(\Sigma_1); \;H^{-1}(\Om) \times L^2(\Om)\right),$$
where  $\delta$ is  a positive constant.

Using \eqref{eq3.39} and \eqref{eq3.44}, we can rewrite  \eqref{subj1} as
\begin{equation} \label{subj2}
\begin{array}{l}
\disp Aw_1\in \{ -v_0(T)+\delta g(T)+B_{H^{-1}(\Om)}(v^1,\rho_1),\;-v_0(T)+B_{L^2(\Om)}(v^0,\rho_0)\}.
\end{array}
\end{equation}
We will show that  $Aw_1$ generates a dense subspace of $H^{-1}(\Om) \times
L^2(\Om)$. For this, let \linebreak $\{ f^0,f^1 \} \in H_{0}^{1}(\Om) \times L^2(\Om)$ and consider the following systems (``adjoint states"):
\begin{equation} \label{eq3.45}
\begin{array}{l}
\varphi'' + L^{\ast}\,\varphi =\disp \alpha_{k}(t) \;\psi \ \mbox{ in } \ Q, \\[5pt]\disp
\varphi = 0 \ \mbox{ on } \ \Sigma, \;\;\;  \varphi(T) = f^0, \ \varphi'(T) = f^1, \;\; y \in \Om,
\end{array}
\end{equation}

\begin{equation} \label{eq3.46}
\begin{array}{l}
\psi'' + L\,\psi = 0 \ \mbox{ in } \ Q,\\[5pt]\disp
\psi = \left\{
\begin{array}{l}
0 \ \mbox{ on } \ \Sigma_1,\\[5pt]\disp
\disp \frac{1}{\sigma \alpha_{k}^2(t)}\,\varphi_{y} \ \mbox{ on } \ \Sigma_2,\\[5pt]\disp
0 \ \mbox{ on } \ \Sigma \backslash \Sigma_0,
\end{array}
\right.\\[5pt]\disp
\psi(0) = \psi'(0) = 0, \;\; y \in \Om.
\end{array}
\end{equation}

Multiplying $\eqref{eq3.46}_1$ by $q$, $\eqref{eq3.45}_1$ by $g$, where $q$, $g$
solve \eqref{eq3.43} and \eqref{eq3.42}, respectively, and integrating in $Q$ we
obtain
\begin{equation} \label{eq3.47}
\int_{0}^{T}\int_{\Om} \alpha_{k}(t)g\,\psi\,dy\,dt =- \frac{1}{\sigma}\int_{\Sigma_2} \frac{1}{\alpha_{k}^4(t)}\,q_{y}\, \varphi_{y} d\Sigma,
\end{equation}
and
\begin{equation} \label{eq3.49}
 \disp \langle g'(T),f^0 \rangle_{H^{-1}(\Om) \times  H_{0}^{1}(\Om)} + \delta \langle g(T), f^0 \rangle_{L^2(\Om) \times  H_{0}^1(\Om)}  - \big( g(T),f^1 \big) =-\int_{\Sigma_1}\frac{1}{\alpha_{k}^2(t)}\,\varphi_{y}\,w_1\,d\Sigma.
\end{equation}

Considering the left-hand side of the above equation as the inner product of $\disp \{g'(T)+
\delta g(T),-g(T)\}$ with $\{ f^0,f^1 \}$ in $ H^{-1}(\Om)  \times L^2(\Om) $ and $
H_{0}^{1}(\Om) \times L^2(\Om)$, we obtain
\begin{equation*}
\Big\langle \big\langle A\,w_1 , f \big\rangle \Big\rangle = - \int_{\Sigma_1}\frac{1}{\alpha_{k}^2(t)}\,\varphi_{y}\,w_1\,d\Sigma,
\end{equation*}
where $\Big\langle \big\langle . , . \big\rangle \Big\rangle$ represent the duality
pairing between $ H^{-1}(\Om) \times L^2(\Om) $ and $ H_{0}^{1}(\Om) \times
L^2(\Om) $. Therefore, if
$$\langle g'(T),f^0 \rangle_{H^{-1}(\Om) \times H_{0}^{1}(\Om)} + \delta \langle g(T), f^0 \rangle_{L^2(\Om) \times H_{0}^1(\Om)} - \big( g(T),f^1 \big) = 0,$$
for all $w_1 \in L^2(\Sigma_1)$, then
\begin{equation} \label{eq3.50}
\disp \varphi_{y}= 0 \ \ \mbox{ on } \ \ \Sigma_1.
\end{equation}
Hence, in case (\ref{decomp 2.A}),
\begin{equation} \label{eq3.51}
\psi = 0 \ \ \mbox{ on } \ \ \Sigma, \;\; \mbox{ so that } \psi\equiv 0.
\end{equation}
Therefore
\begin{equation} \label{eq3.54}
\begin{array}{l}
\disp \varphi'' + L^{\ast}\,\varphi = 0, \;\; \varphi = 0 \mbox{ on }  \Sigma,
\end{array}
\end{equation}
and satisfies \eqref{eq3.50}. Therefore, according to  Holmgren's Uniqueness
Theorem (cf. \cite{LH}; and see also \cite {Cui2} for additional discussions)
and if \eqref{hT} holds, then  $\disp \varphi \equiv 0$, so that $\disp f^0=0, f^1=0$, which completes the proof. \qed  \end{proof}

\section{Optimality System for the Leader}\label{sec5}
Thanks to the results obtained in  Section \ref{sec3}, we can take, for  each
$\disp w_1$, the Nash equilibrium $\disp w_2$ associated with the solution $\disp v$
of \eqref{eq1.14}. We will show the existence of a leader control $\disp w_1$
solution of the following problem:
\begin{equation} \label{eq3.7cil.1}
\disp \inf_{w_1\in \mathcal{U}_{ad}} J(w_1),
\end{equation}
where $\disp \mathcal{U}_{ad}$ is the set of admissible controls
\begin{equation}\label{admcon}
\disp \mathcal{U}_{ad}:=\{w_1\in L^2({\Sigma_1}); \; v \mbox{ solution of } \eqref{eq1.14} \mbox{ satisfying } \eqref{subj1}\}.
\end{equation}
For this, we will use  a duality argument due to Fenchel and Rockafellar \cite{R}
(cf. also \cite{Bre, EK}).

The following result holds:
\begin{thm} \label{teor3.6} Assume the hypotheses $\disp (H1) - (H2)$,
\eqref{decomp 2.A}, \eqref{hT} and \eqref{hT10}  are satisfied. Then for $\{f^0,f^1\}$
in $\disp H_0^1(\Om) \times L^2(\Om)$ we uniquely define $\{\varphi, \psi, v, p \}$
by
\begin{equation} \label{eq3.139}
\begin{array}{l}
\disp \varphi'' + L^* \varphi = \alpha_{k}(t)\psi \ \ \text{in} \ \ Q, \;\;\;  \psi'' + L \psi = 0 \ \ \text{in} \ \ Q, \\[3pt]\disp
\disp v'' + Lv = 0 \ \ \text{in} \ \ Q, \;\;\;  p'' + L^*p  = \alpha_{k}(t)(v - v_2) \ \ \text{in} \ \ Q, \;\;\;  \varphi = 0 \ \ \ \text{on} \ \ \ \Sigma,
\\[3pt]\disp
\disp \psi =
\left\{
\begin{array}{l}
\disp 0 \ \ \text{on} \ \ \Sigma_1, \\[3pt]\disp
\disp \frac{1}{\sigma \alpha_{k}^2(t)}\;\varphi_{y} \ \ \text{on} \ \ \Sigma_2, \\[3pt]\disp
\disp 0 \ \ \ \text{on} \ \ \ \Sigma \backslash \Sigma_0,\\[3pt]
\end{array}
\disp v =
\right.\\
\left\{
\begin{array}{l}
\disp - \frac {1}{\alpha_{k}^2(t)}\;\varphi_{y} \ \ \text{on} \ \ \Sigma_1,\\[3pt]\disp
\disp \frac{1}{\sigma \alpha_{k}^2(t)}\;p_{y} \ \ \text{on} \ \ \Sigma_2,\\[3pt]\disp
\disp 0 \ \ \ \text{on} \ \ \ \Sigma \backslash \Sigma_0,\\[3pt]
\end{array}
\right.\\
\disp p = 0 \ \ \ \text{on} \ \ \ \Sigma, \;\;\; \varphi(.,T) = f^0,\, \varphi'(.,T) = f^1 \ \ \text{in} \ \  \Om, \\[3pt]\disp
\disp v(0) = v'(0) = 0 \ \ \text{in} \ \   \Om, \;\;\;  p(T) = p'(T) = 0 \ \ \text{in} \ \  \Om.
\end{array}
\end{equation}
We uniquely define  $\{f^0,f^1\}$ as the solution of the variational inequality
\begin{equation} \label{eq3.140}
\begin{array}{l}
\disp \big\langle v'(T,f) - v^1, \widehat{f}^0 - f^0\big\rangle_{H^{-1}(\Om) \times  H_{0}^{1}(\Om)} - \big(v(T,f) - v^0, \widehat{f}^1 - f^1\big)  \\[10pt]
\disp + \rho_1\big(||\widehat{f}^0|| - ||f^0||\big) + \rho_0\big(|\widehat{f}^1| - |f^1|\big) \geq 0,\,\forall\, \widehat{f} \in H_{0}^{1}(\Om) \times L^2(\Om).
\end{array}
\end{equation}
Then the optimal leader is given by
\begin{equation*}
w_1 = -\frac{1}{\alpha_{k}^2(t)}\;\varphi_{y} \ \ \text{on} \ \ \Sigma_1,
\end{equation*}
where $\varphi$ corresponds to the solution of \eqref{eq3.139}.
\end{thm}
\begin{proof}
We introduce two convex proper functions as follows, firstly
\begin{equation}\label{eq3.119}
\begin{array}{l}
\disp F_1 : L^2(\Sigma_1) \longrightarrow \mathbb R \cup \{\infty\},\\[5pt]
\disp F_1(w_1) = \frac{1}{2} \int_{\Sigma_1}w_{1}^{2}\,d\Sigma
\end{array}
\end{equation}
the second one
\begin{equation*}
F_2 : H^{-1}(\Omega) \times L^2(\Omega) \longrightarrow \mathbb R \cup \{\infty\}
\end{equation*}
given by
\begin{align} \label{eq3.120}
\nonumber F_2(Aw_1) &= F_2\big(\{g'(T,w_1) + \delta g(T,w_1),-g(T,w_1)\}\big) =\\
& = \left\{
\begin{array}{l}
0, \text{ if }
\left\{
\begin{array}{l}
g'(T) + \delta g(T) \in v^1 - v_0'(T) + \delta g(T) + \rho_1B_{H^{-1}(\Om)},\\
-g(T) \in -v^0 + v_0(T,w_1) - \rho_0B_{L^2(\Om)},
\end{array}
\right.\\
+ \infty, \text{ otherwise}.
\end{array}
\right.
\end{align}
With these notations problems \eqref{inf1}--\eqref{subj1} become equivalent to
\begin{equation} \label{eq3.122}
\begin{array}{l}
\disp \inf_{w_1 \in L^2(\Sigma_1)}\big[F_1(w_1) + F_2(Aw_1)\big]
\end{array}
\end{equation}
provided we prove that the range of $\disp A$ is dense in $\disp H^{-1}(\Om)
\times L^2(\Om)$, under conditions \eqref{hT} and  \eqref{hT10}.

By the Duality Theorem of Fenchel and Rockafellar \cite{R} (see also \cite{Bre, EK}), we have
\begin{equation} \label{eq3.124}
\begin{array}{l}
\inf_{w_1 \in L^2(\Sigma_1)}[F_1(w_1) + F_2(Aw_1)]\\[5pt]\disp  = -\inf_{(\widehat{f}^0,\widehat{f}^1) \in H_{0}^{1}(\Om) \times L^2(\Om)} [F_{1}^{*}\big(A^*\{\widehat{f}^0,\widehat{f}^1\}\big) + F_{2}^{*}\{-\widehat{f}^0, -\widehat{f}^1\}],
\end{array}
\end{equation}
where $\disp F_i^*$ is the conjugate function of $\disp F_i  (i=1,2)$ and $\disp A^*$ the adjoint of $\disp A$.

We have
\begin{equation} \label{eq3.121}
\begin{array}{ccccc}
A^* \ : & \! H_{0}^{1}(\Omega) \times L^2(\Omega) & \! \longrightarrow & \! L^2(\Sigma_1) \\
& \! (f^0,f^1) & \! \longmapsto & \! A^*f = & \! -\dfrac{1}{\alpha_{k}^2(t)}\,\varphi_{y},
\end{array}
\end{equation}
where $\varphi$ is given in \eqref{eq3.45}.
We see easily that
\begin{equation} \label{eq3.125}
F_{1}^{*}(w_1) = F_1(w_1)
\end{equation}
and
\begin{equation}\label{eq3.125.2}
\begin{array}{l}
\disp  F_{2}^{*}(\{\widehat{f}^0,\widehat{f}^1 \})  = \big( v_0(T) - v^0 ,\widehat{f}^1\big) + \langle v^1 - v_0'(T) + \delta g(T), \widehat{f}^0\rangle_{H^{-1}(\Omega) \times H_{0}^{1}(\Omega)}\\ [10pt]
\disp + \rho_1||\widehat{f}^0|| + \rho_0|\widehat{f}^1|.
\end{array}
\end{equation}
Therefore the (opposite of) right-hand side of (\ref{eq3.124}) is given by
\begin{align} \label{eq3.127}
\disp
& - \inf_{\widehat{f} \in H_{0}^{1}(\Omega) \times L^2(\Omega)} \bigg\{\frac{1}{2}\int_{\Sigma_1}\left(\frac{1}{\alpha_{k}^2(t)}\right)^2 \widehat{\varphi}_{y}^2d\;\Sigma  + \big( v^0 - v_0(T) ,\widehat{f}^1\big) \\
\nonumber & - \langle v^1 - v_0'(T) + \delta g(T), \widehat{f}^0\rangle_{H^{-1}(\Omega) \times H_{0}^{1}(\Omega)} + \rho_1||\widehat{f}^0|| + \rho_0|\widehat{f}^1|\bigg\}.
\end{align}
This is the dual problem of \eqref{inf1}, \eqref{subj1}.

We have now two ways to derive the optimality system for the leader control,
starting from the primal or from the dual problem. \qed
\end{proof}

\section{Conclusions}\label{sec6}

In this paper, we have investigated the question of hierarchical control for the one-dimensional wave equation employing the Stackelberg strategy in the case of time dependent domains.

The main achievements of this paper are the existence and uniqueness of Nash equilibrium, the approximate controllability  with respect to the leader control, and the optimality system for the leader control.

As a future work we are looking for improvements and generalizations of these results to other models. To close this section, we make some comments and briefly discuss some possible extensions of our results and also indicate open issues on the subject.

\textbf{$\bullet$} It seems natural to expect that the controllability holds when the speed of the moving endpoint is positive and less than $1$. However,
we did not have success in extending the approach developed in Theorem~\ref{AC} for this case.

\textbf{$\bullet$} In the problem present in Section 2 the domain grows linearly in time (function alpha).  It would be quite interesting to study the controllability for  \eqref{eq1.3} in the non-cylindrical domain  $\widehat{Q}$ when the speed of the moving endpoint is greater than  $1.$ However, it seems very difficult and remains to be done. For interested readers on this subject, we cite for instance \cite {Cui2},\cite {Cui1}, and \cite {Cui}.

\begin{acknowledgements}
\leading{12pt} The author wants to express his gratitude to the anonymous reviewers for their questions and commentaries;
they were very helpful in improving this article.
\end{acknowledgements}

\end{document}